\documentclass[english]{smfart}
\usepackage[francais,english]{babel}
\usepackage{smfthm}
\usepackage{verbatim}
\usepackage{color}
\usepackage{multicol}

\usepackage[letterpaper, margin=1in]{geometry}

\newcommand{\Z}{\mathbb{Z}}
\newcommand{\C}{\mathbb{C}}
\usepackage{euscript}
\newcommand{\scr}[1]{\EuScript{#1}}
\newcommand{\0}{\mathbf{0}}

\newcommand{\TkC}{T_k\langle\scr{C}\rangle}
\newcommand{\WkC}{W_k\langle\scr{C}\rangle}

\author{Nick Ramsey}
\address{Department of Mathematics, DePaul University}
\email{nramsey@depaul.edu}
\title[]{Perturbing Subshifts of Finite Type: Two Words}

\begin{abstract}
	We bound the change in entropy incurred by an irreducible subshift of finite type upon perturbing it by forbidding a pair of admissible words.  Lind has proven such bounds in the one-word case, and we adapt his methods.  In particular, we introduce multi-word correlation polynomials and study their size, as well as that of their determinant in the two-word case.
\end{abstract}

\begin{document}
\frontmatter
\maketitle
\tableofcontents

\theoremstyle{plain}

\section{Introduction}
Let $\scr{A} = \{R_1,\dots, R_r\}$ be an alphabet of $r$ symbols and let $T$ be an $r\times r$ matrix with entries in $\{0,1\}$.  These data determine a subshift of finite type as the collection $\Sigma_T$ of all bi-infinite strings $(s_i)$ in the alphabet $\scr{A}$ such that the $(s_{i+1},s_i)$ entry in $T$ is equal to $1$ for all $i\in\Z$.  Alternatively, let $G_T$ denote the directed graph with vertex set $\scr{A}$ such that there is an edge from $R_i$ to $R_j$ if and only if  the $(j,i)$-entry in $T$ is equal to $1$.  Then $\Sigma_T$ can be interpreted as the set of bi-infinite walks on $G_T$.  We let $\sigma:\Sigma_T\to \Sigma_T$ denote the shift map $\sigma((s_i)) = (s_{i+1})$.  The goal of this note is to explore the entropy of $(\Sigma_T,\sigma)$ and how it is affected by perturbations obtained by forbidding various words from $\Sigma_T$.  

Let $k$ be a positive integer.  By an {\it admissible word of length $k$} (or an {\it admissible $k$-word}) in the alphabet $\scr{A}$, we shall mean one of the form $w = a_1\cdots a_k$ such that $T_{a_{i+1},a_i}=1$ for all $i=1,\dots,k-1$.   We denote by $V_k$ the vector space spanned by such words, and will denote the basis vector associated to the word above by $[w]=[a_1\cdots a_k]$.  Define an injective linear transformation $$\psi_k:V_1\longrightarrow V_k$$ by setting $\psi_k([a])$ equal to the sum of all admissible $k$-words beginning with $a$.  The matrix $T$ defines a linear transformation $T:V_1\to V_1$ given by $$T([a]) = \sum_{T_{b,a}=1} [b].$$  

Given a word $w$, we denote by $\beta w$ and $\eta w$ the words obtained by deleting the last and first symbol of $w$, respectively.  
Let $T_k:V_k\longrightarrow V_k$ denote the linear transformation defined by setting $$T_k([a_1a_2\cdots a_k]) = \sum_{\beta w = a_2\cdots a_k} [w] = \sum [a_2\cdots a_k*].$$  
\begin{lemm}
	The map $\psi_k$ intertwines the maps $T$ and $T_k$ in the sense that $\psi_k\circ T = T_k\circ \psi_k$.
\end{lemm}
\begin{proof}
	Evaluating either side at $[a]$ results in the sum of all admissible $k$-words whose first symbol $b$ evolves from $a$ ({\it i.e.} $T_{b,a}=1$).  
\end{proof}
In particular, $\psi_k(V_1)$ is $T_k$-invariant and the characteristic polynomial of $T_k$ on $\psi_k(V_1)$ coincides with that of $T$ on $V_1$.  The next lemma shows that this is the extent of the ``interesting part'' of the characteristic polynomial of $T_k$.
\begin{lemm}\label{nilpotent}
	The linear transformation $T_k$ is nilpotent on $V_k/\psi_k(V_1)$. 
\end{lemm}
\begin{proof}
	One simply notes that $(T_k)^{k-1}[a_1\cdots a_k] = \psi_k([a_k])$.
\end{proof}
\begin{lemm}\label{absorbing}
	Let $W$ be a $T_k$-invariant subspace with $\psi_k(V_1)\subseteq W\subseteq V_k$ and let $w$ be an admissible $k$-word.  There exists a positive integer $d$ with $T_k^d[w]\in W$, and if $d$ is the smallest such integer, then the vectors $[w], T_k[w], \dots, T_k^{d-1}[w]$ are linearly independent modulo $W$.
\end{lemm}
\begin{proof}
	The existence of $d$ follows from Lemma \ref{nilpotent}.  Suppose that 
	\begin{equation}\label{abslincomb}
	\alpha_0[w]+\alpha_1T_k[w]+\cdots+\alpha_{d-1}T_k^{d-1}[w] \in W
	\end{equation}
	 and let $i$ be the smallest index with $\alpha_i\neq 0$.    Applying $T_k^{d-i}$ to this implies $$\alpha_iT_k^d[w]\in W$$ which contradicts the minimality of $d$.  Thus (\ref{abslincomb}) implies that $\alpha_0=\alpha_1=\cdots=\alpha_{d-1}=0$.
\end{proof}

Consider an admissible $(k+1)$-word $w = a_0a_1\cdots a_k$ that we wish to forbid from occurring in $\Sigma_T$.  This condition is easy to specify from the point of view of $V_k$ and $T_k$, namely, we forbid the shift transition from the initial $k$-word $\beta w = a_0a_1\cdots a_{k-1}$ to the final $k$-work $\eta w = a_1a_2\cdots a_k$.  Since $w$ was assumed to be admissible, the matrix of $T_k$ with respect to the standard basis of $V_k$ has a $1$ in the $(\eta w,\beta w)$ entry, which we must switch to a $0$.  Let $E_w$ denote the endomorphism of $V_k$ whose matrix with respect to the standard basis of admissible $k$-words consists of all $0$s except a $1$ in the $(\eta w,\beta w)$ position.  Now, given a collection $\scr{C}$ of admissible $(k+1)$-words, the matrix of the subshift obtained from $\Sigma_T$ by forbidding words in $\scr{C}$ is $$T_k\langle\scr{C}\rangle = T_k - \sum_{w\in\scr{C}}E_w.$$  

To study the entropy of this subshift, we must compute the characteristic polynomial of  this matrix.  The space $\psi_k(V_1)$ is no longer invariant under $T_k\langle\scr{C}\rangle$, but we can enlarge it slightly so as to obtain an invariant subspace modulo which $\TkC$ is nilpotent.  Since the image of $E_w$ is spanned by $[\eta w]$, the natural space to consider is the minimal $T_k$-invariant subspace of $V_k$ containing $\psi(V_1)$ and the $[\eta w]$ for $w\in\scr{C}$.  This space, which we will denote by $W_k\langle\scr{C}\rangle$, is simply the span of $\psi_k(V_1)$ and the vectors $T_k^i[\eta w]$ for $i\geq 0$ and $w\in\scr{C}$.  
\begin{lemm}
	The transformation $\TkC$ leaves the space $W_k\langle\scr{C}\rangle$  invariant and is nilpotent on $V_k/W_k\langle\scr{C}\rangle$.
\end{lemm}
\begin{proof}
	Since each $E_w$ has image in $W_k\langle\scr{C}\rangle$, this follows immediately from Lemma \ref{nilpotent}.
\end{proof}

The task is now to determine a basis of $W_k\langle\scr{C}\rangle$, express the transformation $T_k\langle\scr{C}\rangle$ in terms of it, and use it to compute the characteristic polynomial of $\TkC$.  That words can be determined by proper subwords will play a role in what follows, and we use the following device to help keep track of this.  For an admissible $k$-word $a_1a_2\cdots a_k$, we denote by $h=h(w)$ the smallest non-negative integer with the property that $a_1\cdots a_{h+1}$ uniquely determines the entire word $a_1a_2\cdots a_k$.  Clearly we have $h\leq k-1$.  
\begin{lemm}\label{risd}
	The integer $h$ is the least positive integer satisfying $T_k^{h}[w]\in \psi_k(V_1)$, and moreover we have $$T_k^{h}[w] = \psi_k([a_{h+1}]).$$
\end{lemm}
\begin{proof}
	Suppose that $T^i_k[w] = \psi_k(v)$ with $i\leq k-1$.  All summands in $T^i[w]$ have first symbol equal to $a_{i+1}$, so we must have $v=[a_{i+1}]$.  Comparing both sides of $T^i_k[w]=\psi_k([a_{i+1}])$ shows that a word beginning $a_{i+1}$ must in fact begin $a_{i+1}\cdots a_k$, which is to say that $w$ is determined by $a_1\cdots a_{i+1}$.  The least such integer is $h$, by definition.
\end{proof}

We will require some standard facts from Perron-Frobenius theory, and refer the reader to Section 1.3 of \cite{Kitchens} for proofs.  We suppose throughout that $T$ is irreducible, which is to say that the directed graph $G_T$ is strongly connected.   For such a $T$, there is associated a positive integer $s$ such that the eigenvalues of maximal absolute value are precisely $$\lambda_0, \lambda_0e^{2\pi i/s},\dots, \lambda_0e^{2\pi i (s-1)/s}$$ with $\lambda_0$ a positive real number.  We will refer to these as the {\it dominant eigenvalues} of $T$ and call $\lambda_0$ the {\it Perron-Frobenius eigenvalue}.  Each of the dominant eigenvalues has multiplicity one, so the characteristic polynomial of $T$ satisfies $$\chi_T(t) = (t^s-\lambda_0^s)q(t)$$ with all the roots $\lambda$ of $q(t)$ satisfying $|\lambda|<\lambda_0$.  The perturbations of $T$ we study in this paper needn't be irreducible, but still correspond to non-negative adjacency matrices and as such have a unique positive real eigenvalue (which we still refer to as the {\it Perron-Frobenius} eigenvalue) that dominates all eigenvalues in absolute value.

In Section \ref{results}, we obtain the following bound on the perturbed eigenvalue, as well as some refinements under stronger assumptions on the pair of words in question.
\begin{theo}\label{main}
Suppose that $T$ is irreducible with Perron-Frobenius eigenvalue $\lambda_0>1$.  Given admissible $(k+1)$-words $w_1$ and $w_2$, let $\lambda_1$ denote the Perron-Frobenius eigenvalue of $T_k\langle w_1,w_2\rangle$.  There exists a positive constant $C$ (depending only on $T$) such that $$|\lambda_1-\lambda_0|\leq C\lambda_0^{-k/2}$$
for $k$ sufficiently large.
\end{theo}

\section{One word (Lind)}\label{oneword}

The results and techniques of this section are due to Lind in \cite{Lind}, though we give a self-contained treatment that differs slightly from his in places.
Suppose that $\scr{C}$ consists of the single word $w=a_0a_1\cdots a_k$ and let $d=h(w)$.  By Lemmas \ref{absorbing} and \ref{risd}, the set $$\{\psi_k([a])\ |\ a\in\scr{A}\}\cup \{[\eta w], T_k[\eta w], \dots, T_k^{d-1}[\eta w]\}$$ is a basis of $\WkC$.  In particular, $d$ is the dimension jump from $\psi_k(V_1)$ to $\WkC$, which is why we have given it the name $d$ instead of $h$.  The dimension jumps incurred by forbidding additional words are generally {\it not} given by their associated $h$ value, and our use of the symbol $d$ here is for forward-compatibility to the multi-word situation.

By Lemma \ref{risd}, the matrix of $\TkC$ with respect to this basis is
\begin{equation}
	\left[\begin{array}{c|cccccc}
		T & \0 & \0 & \0 &  \cdots & \0 & \mathbf{e}_{a_{d+1}} \\ \hline 
		\0^T & 0 & 0 & 0 &  \cdots & 0&0 \\ 
		\0^T & 1 & 0 & 0 &  \cdots & 0&0	\\ 
		\0^T & 0 & 1 & 0 &  \cdots & 0 &0\\
		\0^T & 0 & 0 & 1 & \cdots & 0 &  0 \\
		\vdots & \vdots & \vdots & & \ddots & \vdots& \vdots \\
		\0^T & 0 & 0 & & \cdots & 1 & 0
		\end{array}\right]
\end{equation}
Here, for $a\in\scr{A}$, we let $\mathbf{e}_a$ denote the $r$-dimensional column vector with a $1$ in the $a$ position and $0$s elsewhere.  The matrix of $E_w$ with respect to our basis is concentrated in the row corresponding to the basis vector $[\eta w]$ and has a $1$ in every column whose associated basis vector has $[\beta w]$ in its support.  For $a\in\scr{A}$, the basis vector $\psi_k([a])$ contains $[\beta w]$ in its support if and only if $w$ begins with $a$.  The situation for the remaining  basis vectors $T_k^{i-1}[\eta w]$ is more complicated, and is related to how the word $w$ overlaps itself.  Indeed, $[\beta w]$ occurs in $$T_k^{i-1}[\eta w] = \sum_{\beta^{i-1}w'=a_ia_{i+1}\cdots a_k}[w'] = \sum [a_{i}a_{i+1}\cdots a_k*\cdots*]$$ precisely if 
\begin{equation}\label{overlap}
a_0a_1\cdots a_{k-i} = a_{i}a_{i+1}\cdots a_k
\end{equation}
  For $1\leq i\leq d$, we let $c_{d-i}=1$ if this is so and $c_{d-i}=0$ otherwise, so that the matrix of $E_w$ with respect to our basis is 
\begin{equation}
	\left[\begin{array}{c|ccccc}
		\0 & \0 & \0 & \0 &  \cdots & \0  \\ \hline 
		\mathbf{e}_{a_0}^T & c_{d-1} & c_{d-2} & c_{d-3} &  \cdots & c_{0} \\ 
		\0^T & 0 & 0 & 0 &  \cdots & 0\\ 
		\0^T & 0 & 0 & 0 & \cdots & 0  \\
		\vdots & \vdots & \vdots & & \ddots & \vdots \\
		\0^T & 0 & 0 & & \cdots & 0 
		\end{array}\right]
\end{equation}
Subtracting, we see that the characteristic polynomial we seek is the determinant of 
\begin{equation}
	\left[\begin{array}{c|cccccc}
		T-t & \0 & \0 & \0 &  \cdots & \0 & \mathbf{e}_{a_{d+1}} \\ \hline 
		 -\mathbf{e}_{a_0}^T & -c_{d-1}-t & -c_{d-2} & -c_{d-3} &  \cdots & -c_{1} & -c_{0} \\ 
		\0^T & 1 & -t & 0 &  \cdots & 0&0	\\ 
		\0^T & 0 & 1 & -t &  \cdots & 0 &0\\
		\0^T & 0 & 0 & 1 & \cdots & 0 &  0 \\
		\vdots & \vdots & \vdots & & \ddots & \vdots& \vdots \\
		\0^T & 0 & 0 & & \cdots & 1 & -t
		\end{array}\right]
\end{equation}
The square submatrix with $1$s down the diagonal has full rank and can be exploited using row and column operations to clear out the row below the horizontal line, followed by the bottom-right corner, yielding the matrix 
\begin{equation}
	\left[\begin{array}{c|cccccc}
		T-t & \0 & \0 & \0 &  \cdots & \0 & \mathbf{e}_{a_{d+1}} \\ \hline 
		 -\mathbf{e}_{a_0}^T & 0 & 0 & 0 &  \cdots & 0 & -p(t) \\ 
		\0 & 1 & -t & 0 &  \cdots & 0&0	\\ 
		\0 & 0 & 1 & -t &  \cdots & 0 &0\\
		\0 & 0 & 0 & 1 & \cdots & 0 &  0 \\
		\vdots & \vdots & \vdots & & \ddots & \vdots& \vdots \\
		\0 & 0 & 0 & & \cdots & 1 & 0
		\end{array}\right]
\end{equation}
where $p(t)$ is the \emph{correlation polynomial} $$p(t) = t^d+c_{d-1}t^{d-1}+\cdots+c_1t+c_0 = \sum_{i=0}^dc_{d-i}t^{d-i}$$ and we agree that $c_d=1$ and think of this as accounting for the trivial full self-overlap of $w$.  Permuting the columns, we see that the determinant of this matrix is equal (up to a sign) to 
\begin{equation}\label{charpolyrel}
\det\left[\begin{array}{c|cccccc}
		T-t & \mathbf{e}_{a_{d+1}} & \0 & \0 &  \cdots & \0 & \0 \\ \hline 
		 -\mathbf{e}_{a_0}^T & -p(t) & 0 & 0 &  \cdots & 0 & 0 \\ 
		\0^T & 0& 1 & -t & \cdots &  \cdots & 0	\\ 
		\0^T & 0 & 0 & 1 &  \cdots & 0 &0\\
		\0^T & 0 & 0 & 0 & \cdots & 0 &  0 \\
		\vdots & \vdots & \vdots & & \ddots & \vdots& \vdots \\
		\0^T & 0 & 0 & 0 & \cdots & 1 & -t\\
		\0^T & 0 & 0 & 0 & \cdots & 0 & 1
		\end{array}\right] = \det\left[\begin{array}{c|c}
		T-t & \mathbf{e}_{a_{d+1}} \\ \hline 
		 -\mathbf{e}_{a_0}^T & -p(t)  \\ 
		\end{array}\right] = -p(t)\chi_T(t)\pm M_{a_{d+1};a_0}(t)
\end{equation}
where $\chi_T(t)=\det(T-t)$ and $M_{b;a}$ denotes the minor of $T-t$ obtained by deleting column $a$ and row $b$.  The last equality follows easily by expansion along the last row or column, but is also a special case of the general determinant lemma of the appendix.

Let $\lambda_0$ and $\lambda_1$ denote the Perron-Frobenius eigenvalues of $T$ and of $T_k\langle w\rangle$, respectively.  Lind's approach to the problem of bounding $\lambda_1$ is to first use Rouch\'e's Theorem to prove that, as $k$ grows, the characteristic polynomial of $T_k\langle w\rangle$ has a root close to $\lambda_0$, and then to bootstrap from this to bound the difference $|\lambda_1-\lambda_0|$ in terms of $k$.  In fact, Lind gives a lower bound for this difference as well as an upper bound, but we will only deal with the latter here.  Lind is also working under tighter assumptions on the word than we have imposed, namely that $h(w)=k-1$.  We will require $d\to\infty$ in our estimates below, though we note that since $T$ is irreducible, this is equivalent to $k\to\infty$ by the following observation.
\begin{lemm}\label{hgrows}
	Suppose that $T$ is irreducible and $G_T$ is not a cycle.  For any admissible $k$-word $w$, we have $$k-r\leq h(w)\leq k-1.$$
\end{lemm}
\begin{proof}
	That $h(w)\leq k-1$ is clear.  Since $G_T$ is strongly connected and not a cycle, it has vertex of out-degree at least two, and any vertex can be connected to it in fewer than $r$ steps.  It follows that $a_{h+1}$ cannot uniquely determine the rest of the word $a_{h+1}\cdots a_k$ if $k-h\geq r+1$. 
\end{proof}
\begin{prop}\label{pgrows}
	Let $\rho>1$.  There exists a positive constant $D$ such that $$|p(t)|\geq D|t|^d$$ holds on $|t|\geq\rho$, for all words $w$ with $d$ sufficiently large.
\end{prop}
\begin{rema}
	Estimates of this type recur throughout the paper.  Whenever we refer to a ``constant'' in this context, we mean to say that it depends only on the original shift $T$ and $\rho$ (which will itself depend on $T$ in the sequel).   In particular, such constants are independent of the words in $\scr{C}$ and any of their features like $h$ or $d$.  This comment applies not only to the visible constants such as $D$ here, but also to the implied constant in the phrase ``sufficiently large.''
\end{rema}
The proof of this proposition requires that we analyze the periodic structure of $w$ and ultimately leads us to consider two cases: large period and small period (relative to $d$). 
\begin{defi}
	The {\it fundamental period} of $w$ is the smallest positive integer $p$ with $c_{d-p}=1$.  If no such integer exists, we set $p=k+1$.
\end{defi}
A word $w$ with fundamental period $p$ is the self-concatenation $$w=BBB\cdots BB^*$$ of single block of length $p$, perhaps with a truncated copy $B^*$ of $B$ at the end.
\begin{lemm}\label{Bcopy}
	If a complete copy of $B$ occurs beginning at letter $a_i$ in the word $w$, then $p|i$.
\end{lemm}
\begin{proof}
	The periodicity of $w$ implies that $p$-translates of any occurrence of $B$ are also occurrences.  If $r$ denotes the remainder of $i$ upon division by $p$, the result is that $B$ occurs beginning at $a_r$ as well.  Now $0<r<p$ would contradict the minimality of $p$, so we conclude that $r=0$.
\end{proof}
\begin{lemm}
	If $c_{d-i}=1$ we have either $p|i$ or $i\geq k+2-p$.
\end{lemm}
\begin{proof}
	Suppose that $c_{d-i}=1$.  Looking at (\ref{overlap}), we see that if $k-i+1\geq p$ then there is a complete copy of $B$ beginning at $a_i$, and the result now follows from Lemma \ref{Bcopy}.
\end{proof}
The upshot of these Lemmas is that the polynomial $p(t)$ takes the form $$p(t) = t^d+t^{d-p}+t^{d-2p}+\cdots+t^{d-Mp}+\psi(t)$$ where $$M = \left\lfloor \frac{d}{p}\right\rfloor$$ and $\psi(t)$ has degree at most $$d-(k+2-p)\leq k-1-(k+2-p)=p-3.$$  Lind's idea to get a lower bound is essentially to use this to write $$p(t) = D(t)+E(t)$$ as a dominant terms plus an error term and then bound $D(t)$ from below and bound $E(t)$ from above.  How $D(t)$ and $E(t)$ are chosen depends on the size of $p$ relative $d$.  In the estimates that follow, we fix $\rho>1$ and assume that $t\in\C$ satisfies $|t|\geq \rho$.

Fix $\alpha\in(0,1)$, which will function as a small/large cutoff for $p$ relative to $d$.  Any such $\alpha$ will do, though we find it clarifying to leave it as unspecified rather than fix a particular value, {\it e.g.} $\alpha=1/2$.
\begin{proof}[Proof of Proposition \ref{pgrows}]
Suppose first that $p\leq\alpha d$.   Here, we take $$D(t) = t^d+t^{d-p}+\cdots+t^{d-Mp} =t^d\frac{t^{p}-t^{-Mp}}{t^p-1}$$ and $E(t)=\psi(t)$.  We have
$$|D(t)| \geq |t|^d\frac{|t|^{p}-|t|^{-Mp}}{|t|^p+1} \geq |t|^d\frac{|t|^{p}-1}{|t|^p+1}\geq  |t|^d\left(\frac{\rho-1}{\rho+1}\right)$$
As for the error term, note that since all coefficients of $\psi(t)$ are $0$ or $1$, we have 
$$|E(t)|\leq 1+|t|+\cdots+|t|^{p-3} = \frac{|t|^{p-2}-1}{|t|-1}\leq\frac{1}{\rho-1}|t|^{p-2}\leq \frac{1}{\rho^2(\rho-1)}|t|^{\alpha d}$$
It follows that $$|p(t)| \geq |t|^d\left(\frac{\rho-1}{\rho+1}-\frac{1}{\rho^2(\rho-1)} |t|^{(\alpha-1)d}\right) \geq |t|^d\left(\frac{\rho-1}{\rho+1}- \frac{ \rho^{(\alpha-1)d}}{\rho^2(\rho-1)}\right) \geq |t|^d\cdot \frac{\rho-1}{2(\rho+1)}$$ for $d$ sufficiently large.

Now suppose that $p\geq \alpha d$.    Here, we simply take $D(t) = t^d$ and let $E(t)$ consist of the non-leading terms of $p(t)$.  We have 
$$|E(t)|\leq |t|^{d-p}+|t|^{d-p-1}+\cdots+|t|+1 = \frac{|t|^{d-p+1}-1}{|t|-1}\leq \frac{1}{\rho-1}|t|^{d-p+1}\leq \frac{1}{\rho-1}|t|^{(1-\alpha)d+1}$$
so $$|p(t)| \geq |t|^d-\frac{1}{\rho-1}|t|^{(1-\alpha)d+1} = |t|^d\left(1-\frac{1}{\rho-1}|t|^{-\alpha d+1}\right) \geq |t|^d\left(1-\frac{1}{\rho-1}\rho^{-\alpha d+1}\right)
\geq \frac{1}{2}|t|^d$$ for $d$ sufficiently large.  The proposition follows by taking $D$ to be the smaller of the constants obtained in the two cases and by taking ``sufficiently large'' to mean at least the larger of the implied constants in each case.  
\end{proof}

\begin{prop}\label{onewordqualbound}
	Suppose that $\rho>1$ satisfies $|\lambda|<\rho<\lambda_0$ for all non-dominant eigenvalues $\lambda$ of $T$.   For $d$ sufficiently large, we have $\lambda_1\geq \rho$.
\end{prop}
\begin{proof}
	Let $X(t)$ denote the characteristic polynomial of $T_k\langle w\rangle$.  We have seen in (\ref{charpolyrel}) that, up to a sign, $X(t)$ differs from $p(t)\chi_T(t)$ by a minor $M(t)$ of $T-t$.  As there are only finitely many such minors, they are collectively bounded by a single constant on the compact set $|t|=\rho$.  The hypothesis on $\rho$ implies that $\chi_T(t)$ is nonvanishing on $|t|=\rho$, and hence is bounded below by a nonzero constant.  Now Proposition \ref{pgrows} implies that $$|p(t)\chi_T(t)|>|M(t)|\ \ \  \mbox{on}\ \ |t|=\rho$$  for $d$ sufficiently large, and Rouch\'e's Theorem implies that $X(t)$ and $p(t)\chi_T(t)$ have the same number of roots in $|t|<\rho$ for such $d$.  Since these two polynomials have equal degree, they must have the same number of roots with $|t|\geq \rho$, and in particular we must have $\lambda_1\geq \rho$.
\end{proof}

\begin{theo}\label{Lindmain}
	Suppose that $T$ is irreducible with $\lambda_0>1$.   There exists a constant $C$ such that $$|\lambda_1-\lambda_0| < C\lambda_0^{-d}$$ for $d$ sufficiently large.
\end{theo}
\begin{proof}
	Choose $\rho>1$ with $|\lambda|<\rho<\lambda_0$ for all non-dominant eigenvalues $\lambda$ of $T$.  
	Write $\chi_T(t) = (t^s-\lambda_0^s)q(t)$ as in Section 1 and plug $\lambda_1$ into (\ref{charpolyrel}) to see $$|\lambda_1^s-\lambda_0^s| = \frac{|M_{a_{d+1};a_0}(\lambda_1)|}{|p(\lambda_1)|\cdot |q(\lambda_1)|}$$  Finiteness of the collection of minors and the fact that $\lambda_1\leq \lambda_0$ implies a 
universal upper bound for the numerator.  The choice of $\rho$, Proposition \ref{onewordqualbound}, and the discussion at the end of Section 1 imply a universal nonzero lower bound on $|q(\lambda_1)|$ for $d$ sufficiently large.  Finally, Proposition \ref{pgrows} and Proposition \ref{onewordqualbound} imply that there is a positive constant $C_0$ depending only on $T$ and $\rho$ such that 
\begin{equation*}
|\lambda_1^s-\lambda_0^s|\leq C_0\lambda_1^{-d}
\end{equation*}
 for $d$ sufficiently large.   Since $1<\lambda_1\leq\lambda_0$, we have 
 \begin{equation}\label{almost}
 |\lambda_1-\lambda_0| \leq |\lambda_1^s-\lambda_0^s|\leq C_0\lambda_1^{-d}
 \end{equation}
 for such $d$.

Finally, we bootstrap from this as in \cite{Lind} by noting that the differentiability of $\log$ implies that there exists $s>0$ such that $$\log(\lambda_1)\geq \log(\lambda_0)+s(\lambda_1-\lambda_0)$$ for $\lambda_1-\lambda_0$ sufficiently small.  Thus $$\lambda_1^{-d}= \lambda_0^{-d\frac{\log(\lambda_1)}{\log(\lambda_0)}}\leq \lambda_0^{-d}\cdot \lambda_0^{-ds\frac{\lambda_1-\lambda_0}{\log(\lambda_0)}}$$  The second factor is bounded as $d\to \infty$ by (\ref{almost}), which gives $$|\lambda_1-\lambda_0|< C\lambda_0^{-d}$$ for some constant $C$ and sufficiently large $d$, as desired.

\end{proof}

\section{Two words: Structure of $\WkC$ and $T_k$}
Now suppose that $\scr{C}$ consists of a pair $w_1=a_0a_1\cdots a_k$ and $w_2= b_0b_2\cdots b_k$ of admissible $(k+1)$-words.  We build up a basis of $W_k\langle\scr{C}\rangle$ beginning with the single word $w_1$ as in the previous section: let $d_1= h(w_1)$ be as above, so that the set $$\{\psi_k([a])\ |\ a\in\scr{A}\}\cup \{[\eta w_1], T_k[\eta w_1], \dots, T_k^{d_1-1}[\eta w_1]\}$$ is linearly independent and spans the subspace $W_k\langle w_1\rangle$ of $W_k\langle w_1,w_2\rangle$.

Now we bring in $w_2$.  Let $d_2$ denote the minimal non-negative integer for which $T_k^{d_2}[\eta w_2]\in W_k\langle w_1\rangle$.  By Lemma \ref{absorbing}, the set 
\begin{equation}\label{basis}
\{\psi_k([a])\ |\ a\in\scr{A}\}\cup \{[\eta w_1], T_k[\eta w_1], \dots, T_k^{d_1-1}[\eta w_1]\} \cup \{[\eta w_2], T_k[\eta w_2], \dots, T_k^{d_2-1}[\eta w_2]\}
\end{equation}
 is a basis of $\WkC$.  In order to determine the matrix of $T_k$ with respect to this basis, we must explicate both $T_k^{d_1}[\eta w_1]$ and $T_k^{d_2}[\eta w_2]$.  The first of these is as in the previous section: $$T_k^{d_1}[\eta w_1] = \psi_k([a_{d_1+1}]).$$  The situation for $w_2$ depends on how the words $w_1$ and $w_2$ interact. Clearly we have $d_2\leq h_2$ with equality if and only if we have $T_k^{d_2}[\eta w_2]\in \psi_k(V_1)$, in which case $$T_k^{d_2}[\eta w_2] = \psi_k([b_{d_2+1}])$$ as in the one-word situation.    

Suppose that $d_2<h_2$ and set $\delta=h_2-d_2$.  This implies some sort of nontrivial interaction between $w_1$ and $w_2$, which we now explore.  We have 
\begin{equation}\label{origrelation}
T_k^{d_2}[\eta w_2] = \alpha_0[\eta w_1]+\alpha_1T_k[\eta w_1]+\cdots+\alpha_{d_1-1}T_k^{d_1-1}[\eta w_1] + \psi_k(v)
\end{equation}
 with not all coefficients $\alpha_i$ equal to $0$.  Applying $T_k^{\delta}$ to both sides and absorbing $T_k^{h_2}[\eta w_2]$ and $T_k^i[\eta w_1]$ for $i\geq d_1$ into $\psi_k(V_1)$ we see $$\alpha_0T_k^{\delta}[\eta w_1]+\cdots+\alpha_{d_1-1-\delta}T_k^{d_1-1}[\eta w_1]\in \psi_k(V_1)$$ which forces $$\alpha_0=\cdots = \alpha_{d_1-1-\delta}=0$$ by Lemma \ref{absorbing}.  Now applying $T_k^{\delta-1}$ and reasoning similarly we see $$T_k^{h_2-1}[\eta w_2] \equiv \alpha_{d_1-\delta}T_k^{d_1-1}[\eta w_1] \pmod{\psi_k(V_1)}$$ which implies that $\alpha_{d_1-\delta}\neq 0$.  Let $b=b_{d_2+1}$, and note that each term in the relation (\ref{origrelation}) is either fixed or killed by projection onto the ``first symbol is $b$ subspace.''  Since this clearly fixes the left-hand side, it must in fact fix every term by uniqueness of this linear relation.  It follows that $v=\gamma [b]$ for some $\gamma$.  Thus our relation above takes the form 
\begin{equation}\label{bigrelation}
T_k^{d_2}[\eta w_2] = \alpha_{i_1}T_k^{i_1}[\eta w_1]+\cdots+\alpha_{i_n}T_k^{i_n}[\eta w_1]+ \gamma\psi_k([b])
\end{equation}
 where $i_1=d_1-\delta$ and we have retained only the nonzero $\alpha$-coefficients.  

Let $S_0$ denote the support of $T_k^{d_2}[\eta w_2]$, which is precisely the set of admissible $k$-words beginning with $$s_0:=b_{d_2+1}\cdots b_{h_2}b_{h_2+1}$$ since the rest of $w_2$ is then forced.  Similarly, let $S_m$ denote the support of $T_k^{i_m}[\eta w_1]$, namely the set of words beginning $$s_m:=a_{i_m+1}\cdots a_{h_1}a_{h_1+1}.$$  Looking at (\ref{bigrelation}), we see that $s_1$ is the concatenation $$s_1=B_1B_2\cdots B_{n-1}B_n$$ of $n$ blocks, each of which begins with $b$, namely,  $$B_1=a_{i_1+1}\cdots a_{i_2},\ \  B_2=a_{i_2+1}\cdots a_{i_3},\ \ \dots,\ \ B_n=a_{i_n+1}\cdots a_{h_1+1}.$$  In this notation, the words $s_2, s_3, \dots, s_n$ are obtained by successively dropping blocks off of the left side of $s_1$.

For an admissible word $A$ and a positive integer $e$, let $A^e$ denote the $e$-fold self-concatenation of $A$.  We call an admissible word $B$ {\it simple} if it is not equal to $A^e$ for any word $A$ and $e>1$.  Note that a word that overlaps itself in the manner discussed above can be simple, as the word $abcab$ illustrates.  On the other hand, the following lemma shows that a stronger kind of self-overlap does preclude simplicity.
\begin{lemm}\label{simple}
	Suppose that $B$ occurs nontrivially in $BB$ (that is, not merely at the beginning or end).  Then $B$ is not simple.
\end{lemm}
\begin{proof}
	If $B$ occurs nontrivially in $BB$, it meets the first copy in a nonempty subword $B_1$ and the second copy in a nonempty word $B_2$.  Then $B$ is simultaneously equal to both concatenations $$B=B_1B_2=B_2B_1.$$
	
	We claim that any pair of strings that commute in this fashion must be powers of a common string.  If not, let $w=CD=DC$ be the shortest counterexample.  If $C$ and $D$ are of equal length, then they must coincide and we have $w=C^2$ contrary to our assumption.  Otherwise, we may assume that $C$ is the shorter word and then $CD=DC$ implies that $D=CC'$ for some word $C'$.  Pruning $C$ from the left side of $CD=DC$ then yields $CC'=C'C$.  Minimality of $w$ implies that $C$ and $C'$ are powers of a common string, which implies that $w$ is a power of this string as well, contrary to our initial assumption.
\end{proof}

\begin{prop}\label{bigsummary}
	Suppose that $d_2<h_2$.  One of the following holds.
	\begin{enumerate}
	\renewcommand{\labelenumi}{(\Alph{enumi})}
	\item 
	\begin{enumerate}
	\renewcommand{\labelenumii}{(\roman{enumii})}
	\item The supports $S_0,\dots,S_n$ are pairwise disjoint and collectively exhaust all admissible $k$-words beginning with $b$.  
	\item $\gamma=1$ and $\alpha_{i_1}=\cdots=\alpha_{i_n}=-1$
	\item $B_1=B_2=\cdots=B_{n-1}$ is simple and $B_n$ differs from a truncation of this common block exactly in the last symbol
	\end{enumerate}
	\item 
	\begin{enumerate}
	\renewcommand{\labelenumii}{(\roman{enumii})}
	\item $n=2$
	\item $S_0$ and $S_1$ are disjoint and exhaust $S_2$
	\item $\gamma=0$, $\alpha_{i_1}=-1$, and $\alpha_{i_2}=1$
	\end{enumerate}
	\item 
	\begin{enumerate}
	\renewcommand{\labelenumii}{(\roman{enumii})}
	\item $n=1$
	\item $S_0=S_1$
	\item $\gamma=0$, $\alpha_{i_1}=1$
	\end{enumerate}
	\end{enumerate}
\end{prop}
The proof of this proposition is somewhat long and involved.  The basic strategy throughout is to construct words that begin with $b$ and see where they are obliged to fit into the various supports and how they interact with the block structure of $s_1$.  
\begin{proof}
The lengths of the strings $s_0,\dots,s_n$ satisfy $$\delta+1=\ell(s_0) = \ell(s_1)>\ell(s_2)>\cdots>\ell(s_n).$$  It follows that any nontrivial intersection between the supports $S_m$ is in fact {\it containment}.  In particular, if $S_0$ meets any of $S_1,\dots, S_n$, then $S_0$ is contained in the latter support, while if $S_n$ meets one of $S_0,\dots,S_{n-1}$, then $S_n$ must contain it.  Note also that the linear independence of the $T_k^{i_m}[\eta w_1]$ implies that none of the sets $S_1,\dots, S_n$ is the union of others.

To further pin down the behavior of these supports, we break into two cases.  Suppose first that $\gamma\neq 0$ and rearrange (\ref{bigrelation}) as
\begin{equation}\label{bigrelation2}
\gamma\psi_k([b]) = T_k^{d_2}[\eta w_2] -(\alpha_{i_1}T_k^{i_1}[\eta w_1]+\cdots+\alpha_{i_n}T_k^{i_n}[\eta w_1])
\end{equation}
This implies that every admissible $k$-word beginning with $b$ occurs on the right-hand side with equal coefficient.  It also implies that none of the supports $S_0,S_1,\dots,S_n$ can be the union of others, since this would imply a nontrivial linear relation as above with $\gamma=0$.  We claim that the supports $S_1,\dots,S_n$ are pairwise-disjoint.  Indeed, suppose that $S_p\subseteq S_q$ for some $p<q$ and assume that $q$ is the largest such index and that $p$ is the largest such index for this particular $q$.  Since $S_q$ cannot be a union of any of the $S_0,S_1,\dots,S_{q-1}$, there exists $x\in S_q$ that is contained in none of these supports.  It follows by maximality of $q$ that $S_q$ is the unique support containing $x$.  The support $S_p$ is also not a union of other supports, so there exists $y\in S_p$ that is contained in $S_p$ and $S_q$ and no other supports (by maximality of $q$ and $p$).  Comparing the coefficients of $x$ and $y$ in (\ref{bigrelation2}), we see that they differ by $\alpha_{i_q}$, which is a contradiction since they must be equal and $\alpha_{i_q}\neq 0$.   Thus $S_1,\dots, S_n$ are pairwise disjoint.

Now consider any string $\widehat{s_1}$ that is identical to $s_1$ with the exception of ending in any symbol but $a_{h_1+1}$.   Note that such a string exists since $a_{h_1+1}$ was assumed to be the {\it first} symbol that determines the rest of the word.  Any admissible $k$-word that begins with $\widehat{s_1}$ must occur somewhere on the right-hand side of (\ref{bigrelation2}).   We claim that such a word must occur in $S_0$.   It cannot belong to $S_1$ by construction.  Suppose that it belongs to $S_m$ for some $m\geq 2$.   Then the word $s_m$ would be a truncation of $\widehat{s_1}$, and hence of $s_1$, which implies that $S_1\subseteq S_m$ and contradicts the observation above.  The upshot is that $\widehat{s_1}$ must occur in $S_0$, which is to say that $\widehat{s_1}=s_0$ and in particular implies that $$b_{d_2+1}\cdots b_{h_2} = a_{1+d_1-\delta}\cdots a_{d_1},$$ which is a common subword of length $\delta$.  This description of $s_0$ also implies that $S_0$ is disjoint from each of $S_1,\dots,S_n$ since $s_0$ begins with one of $s_m$ if and only if $s_1$ does as well.  Finally, the fact that the collection of admissible $k$-words that begin with $b$ is the disjoint union of the $S_0,S_1,\dots,S_n$ says that the relation (\ref{bigrelation}) has $\gamma=1$ and $\alpha_{i_m}=-1$ for all $m$.
 
Suppose now that $\gamma=0$.  Here, (\ref{bigrelation}) implies that each support $S_m$ is contained in the union of the remaining supports.  In particular, we see that $S_n$ must {\it be} the union of the remaining supports $S_0,\dots, S_{n-1}$, since it contains any support it meets.   In other words, each of the strings $s_0,\dots, s_{n-1}$ begins with $s_n$, and every string that begins with $s_n$ begins with one of $s_0,\dots,s_{n-1}$.  Now we rewrite (\ref{bigrelation}) as $$\alpha_{i_n}T_k^{i_n}[\eta w_1] = T_k^{d_2}[\eta w_2] - (\alpha_{i_1}T_k^{i_1}[\eta w_1]+ \cdots + \alpha_{i_{n-1}}T_k^{i_{n-1}}[\eta w_1])$$ and proceed exactly as in the previous case to conclude that $s_0$ and $s_1$ differ only in the last symbol, that $S_0,S_1,\dots,S_{n-1}$ are pairwise disjoint and collectively exhaust $S_n$, and finally that $\alpha_{i_n}=1$ and $\alpha_{i_1}=\cdots=\alpha_{i_{n-1}}=-1$.  

Having worked out the nature of the various supports, we turn to a more detailed analysis of the blocks $B_1,\dots,B_n$.  Let us return to the $\gamma\neq 0$ situation.  Consider the word $\overline{B_{n-1}}=B_{n-1}B_{n-1}\cdots$ where we repeat until the length is at least $\delta+1$.  This word begins with $b$, and therefore begins with one of $s_0,s_1,\dots, s_n$.  It cannot begin with $s_n=B_n$, since then $B_n$ would a truncated power of $B_{n-1}$, and hence $s_{n-1}=B_{n-1}B_n$ would also begin with $s_n=B_n$, which would imply $S_{n-1}\subseteq S_n$, contrary to the above work.  If $\overline{B_{n-1}}$ were to begin with one of $s_1,\dots, s_{n-1}$, then the string $s_{n-1}=B_{n-1}B_n$ would occur in $\overline{B_{n-1}}$.  The periodicity of this string would then imply again that $s_n=B_n$ occurs at the beginning of $s_{n-1}=B_{n-1}B_n$, which yields the same contradiction.  We conclude that $\overline{B_{n-1}}$ begins with $s_0$.  Note that, since $s_0$ differs from $s_1$ only at the last symbol, this implies in particular that the string $B_1\cdots B_{n-1}$ is a truncated power of $B_{n-1}$.  The subtlety is that these copies of $B_{n-1}$ do not {\it a priori} line up with the blocks $B_1, B_2, \dots, B_{n-1}$.  
	
The next step is to establish that each of $B_1, \dots, B_{n-1}$ is equal to a power of $B_{n-1}$.  Consider a word beginning $\widehat{s_2}$, which is $s_2$ with the last symbol switched as in $\widehat{s_1}$ above.  Such a word clearly cannot begin $s_2$.  It also cannot begin with any of the words $s_3,\dots,s_n$ since these words are shorter and their occurrence at the beginning of $\widehat{s_2}$ would imply their occurrence at the beginning of $s_2$, which contradicts the disjointness of the supports $S_1,\dots,S_n$.  Thus a word that begins with $\widehat{s_2}$ must begin with either $s_0$ or $s_1$.  Since these words are longer, this means that $\widehat{s_2}$ occurs at the beginning of $s_0$ or $s_1$, and in particular that $\widehat{s_2}$, and hence $s_2$, begins with $B_{n-1}$. The same argument applied to $\widehat{s_3}$ (with evident notation) shows that it must occur at the beginning of $s_0, s_1$, or $s_2$, and hence $s_3$ must also begin with $B_{n-1}$.  Proceeding in this fashion, we conclude that every one of $s_1,\dots,s_{n-1}$ begins with $B_{n-1}$.  

We are still short of the conclusion that each $B_i$ is a power of $B_{n-1}$, owing to the possibility of nontrivial occurrences of $B_{n-1}$ in the string $B_1B_2\cdots B_{n-1}$.  For example, the copy of $B_{n-1}$ that must begin at $B_2$ in this string by the previous paragraph might begin in the middle of a copy of $B_{n-1}$ in $\overline{B_{n-1}}$.  But the fact that the entire string $B_1B_2\cdots B_{n-1}$ is tiled over by copies of $B_{n-1}$ means that such an overlap has $B_{n-1}$ occurring in $B_{n-1}B_{n-1}$ nontrivially.  Lemma \ref{simple} implies that $B_{n-1}=A^c$ for some $c\geq 2$ and some string $A$, which we may take to be simple.  In particular, $B_{n-1}$ is not simple, from which we derive a contradiction as follows.  
The word $\overline{A}=\overline{B_{n-1}}$ begins with $s_0=B_1\cdots B_{n-1}\widehat{s_n}$, so the simplicity of $A$ means that each block in $B_1\cdots B_{n-1}$ is a power of $A$, and $\widehat{s_n}$ is a truncated power of $A$, say $\widehat{s_n}=A^eA^*$ where $A^*$ is a truncation of $A$.  Switching the last symbol back to $a_{h_2+1}$, we see that $B_n=A^e\widehat{A^*}$ with evident notation.   Consider admissible words beginning with $\widehat{A^*}$.  Such a word cannot begin with any of $s_0, s_1, \dots, s_{n-1}$ as these all begin with $A$.  Thus it must begin with $s_n=B_n$, which implies that $e=0$.    Finally, consider an admissible word beginning $A\widehat{A^*}$.  Such a word cannot begin with any of $s_0,s_1,\dots,s_{n-1}$ since these all begin with at least two copies of $A$ (since $B_{n-1}$ has $c\geq 2$ copies of $A$).  Such a word clearly cannot begin with $s_n=B_n=\widehat{A^*}$ either, which gives us a contradiction.

We conclude that $B_{n-1}$ is simple and that such self-overlaps of $B_{n-1}$ do not occur, which is to say that each block $B_1, B_2, \dots, B_{n-1}$ is a power of $B_{n-1}$.    Let $c_i$ be the positive integer with $B_i = B_{n-1}^{c_i}$.  The last step is to show that $c_i=1$ for all $i$.  We have $\widehat{s_n} = B_{n-1}^eB_{n-1}^*$ for some $e$, where $B_{n-1}^*$ is a truncation of $B_{n-1}$.  Switching the last symbol we get $B_n=B_{n-1}^e\widehat{B_{n-1}^*}$ with evident notation.  We have 
\begin{eqnarray*}
	s_0 &=& B_{n-1}^{c_1+\cdots+c_{n-2}+1+e}B_{n-1}^*\\
	s_1 &=& B_{n-1}^{c_1+\cdots+c_{n-2}+1+e}\widehat{B_{n-1}^*} \\
	s_2 &=& B_{n-1}^{c_2+\cdots+c_{n-2}+1+e}\widehat{B_{n-1}^*} \\
	s_3 &=& B_{n-1}^{c_3+\cdots+c_{n-2}+1+e}\widehat{B_{n-1}^*} \\
	 & \vdots & \\
	s_{n-1} &=& B_{n-1}^{1+e}\widehat{B_{n-1}^*}\\
	s_n &=& B_{n-1}^{e}\widehat{B_{n-1}^*}
\end{eqnarray*}
and as usual any admissible $k$-word beginning with $b$ must begin with one of these strings.  Looking for words beginning $\widehat{B_{n-1}^*}$ here immediately yields $e=0$.  Next, we look for words beginning $B_{n-1}^{c_1+\cdots+c_{n-2}}\widehat{B_{n-1}^*}$.  Such words evidently cannot begin with $s_0$ or $s_1$.  They also cannot begin with $s_3,\dots s_n$ since these words are strictly shorter than $B_{n-1}^{c_1+\cdots+c_{n-2}}\widehat{B_{n-1}^*}$, which would imply that $B_{n-1}^{c_1+\cdots+c_{n-2}}B_{n-1}^*$ begins with one of $s_3,\dots,s_n$, again contradicting the disjointness of the supports $S_0,\dots, S_n$.  Thus such words must begin with $s_2$, which implies $c_1=1$.  Looking for words beginning $B_{n-1}^{c_2+\cdots+c_{n-2}}\widehat{B_{n-1}^*}$ implies similarly that $c_2=1$, and proceeding in this manner we conclude that $c_i=1$ for all $i$.  This completes the analysis of the $\gamma\neq 0$ situation, establishing Case (A) of our proposition as holding there.

Suppose again that $\gamma=0$, and suppose that $n\geq 3$.  Recall that we have established that all admissible words of length at least $\delta+1$ that begin with $B_n$ must begin with one of $s_0,s_1,\dots, s_{n-1}$.  Reasoning as with $\overline{B_{n-1}}$ in the case $\gamma\neq 0$, we see that the word $\overline{B_{n-2}}$ must begin with $s_0$.  Once again, the next task is to show that the words $s_1,s_2,\dots,s_{n-2}$ all begin with $B_{n-2}$ by successively considering the words $\widehat{s_2}, \widehat{s_3}, \dots, \widehat{s_{n-2}}$.  Reasoning similarly, we see that any word beginning $\widehat{s_{n-1}}$ must begin with one of $s_0,s_1,\dots,s_{n-2}$, say $s_i$.  If $\widehat{s_{n-1}}$ has length at most $B_{n-2}$, then it is shorter than $s_i$ and the fact that any word beginning $\widehat{s_{n-2}}$ must begin $s_i$ means that the remainder of $s_i$ is determined, contrary to the fact that the last symbol is not determined, by minimality of $h_1$.  Thus $\widehat{s_{n-1}}$ is longer than $B_{n-2}$ and must also begin with it, which implies that $s_{n-1}$ begins with $B_{n-2}$ as well.  

Write $B_{n-2} =A^c$ where $A$ is simple.  Reasoning as in the $\gamma\neq 0$ situation, we again conclude that each $B_1,\dots, B_{n-2}$ is a power of $A$.  Observe that the word $s_{n-1}$ begins with $A^c$, the word $s_{n-2}$ begins with $A^{2c}$, and the words $s_i$ for $i<n-2$ begin with $A^{2c+1}$.  We consider two cases.  Suppose first that $\ell(B_n)\leq \ell(A)$.  
Since $s_0=\widehat{s_1}$ occurs at the beginning of $\overline{B_{n-2}} = \overline{A}$, the last two symbols $b_{h_2}b_{h_2+1}$ occur either in $A$ somewhere, or $b_{h_2}$ occurs at the end of $A$ and $b_{h_2+1}$ occurs at the beginning.  In either case, we can form $A\widehat{A^*}$ by switching the $b_{h_2  +1}$ in the second copy of $A$ and truncating at this point, leaving everything before it alone.  Since $B_n$ occurs at the beginning of $s_{n-2}=B_{n-2}B_{n-1}B_n$ and has length at most $\ell(A)$, it must occur at the beginning of $A\widehat{A^*}$.  Thus a word beginning $A\widehat{A^*}$ must begin with one of $s_0,s_1,\dots,s_{n-1}$.  If $c\geq 2$ then this is impossible since each of these words begins with $A^2$.  Thus $c=1$ and we conclude that a word beginning $A\widehat{A^*}$ must begin $s_{n-1}$.   Write $A=B_nA_1$ and note that $$s_{n-1}=B_{n-1}B_n = B_nA_1B_n$$ and $$s_{n-2} = B_{n-2}B_{n-1}B_n = B_nA_1B_nA_1B_n$$  Thus $s_{n-1}$ occurs at the beginning of $s_{n-2}$, contrary to the disjointness of the supports $S_{n-1}$ and $S_{n-2}$.  

Finally, suppose that $\ell(B_n)>\ell(A)$.  The fact that $s_0=B_1\cdots B_{n-1}\widehat{B_n}$ is tiled over by copies of $A$ and each of $B_1,\dots,B_{n-2}$ is a power of $A$ implies that $B_{n-1}\widehat{B_n}$ is tiled by copies of $A$.  We know that $B_n$ occurs at the beginning of $\overline{A}$ since it occurs at the beginning of $s_0$.  Now the fact that $\ell(B_n)>\ell(A)$ means that a copy of $A$ occurs at the beginning of $B_n$ as well as $\widehat{B_n}$ since these differ only in the last symbol.  It follows that $B_{n-1}=A^e$ for some $e$, since otherwise this copy of $A$ would meet an $AA$ in the tiling of $B_{n-1}\widehat{B_n}$ nontrivially, contrary to Lemma \ref{simple}.  But this implies that $\widehat{B_n}$ is tiled over by copies of $A$, which is ridiculous since we know that $B_n$ is as well.  Having exhausted all other options (as well the reader) we conclude that $n\leq 2$ in the $\gamma=0$ setting, which completes the proof of the proposition.
\end{proof}

\begin{coro}\label{smalloverlap}
	Suppose that $d_2<h_2$ and we are in Case (A) of Proposition \ref{bigsummary} .  Then $c^{11}_{d_1-i}=0$ for $0<i<\delta$.  If the words are ordered so that $h_1\leq h_2$, then we have $c^{21}_{d_1-i}=0$ for $i<\delta-r+1$.  
\end{coro}
\begin{proof}
	Let $B$ denote the block $B_1=B_2=\cdots=B_{n-1}$ of Case (A) of Proposition \ref{bigsummary}, so $s_1=B^{n-1}B_n$.  Recall that $B$ is simple and that $B_n$ is a truncation of $B$ with the last symbol switched.  Let $s\widehat{s'}$ denote the last two symbols in $B_n$ with corresponding $ss'$ in $B$.  
	
	Since $\ell(s_1)=\delta+1$, any $0<i\leq\delta$ with $c^{11}_{d_1-i}=1$ dictates a self-overlap of $s_1$ in $\delta+1-i$ symbols.  Since $h_1\leq h_2$, we have $$d_2+1 = d_1-(h_2-d_2) +(h_2-h_1)+1 = i_1+(h_2-h_1)+1\geq i_1+1.$$  This is to say that the subword $s_0$ of $w_2$ begins $h_2-h_1$ units to the right of the subword $s_1$ of $w_1$.  Since $s_1$ and $s_0$ agree outside of the last symbol, we see that $c^{21}_{d_1-i}=1$ dictates a self-overlap of $s_1$ in $$\delta+1-(h_2-h_1)-i \geq \delta+1-(r-1)-i$$ symbols, where we have used Lemma \ref{hgrows}.    Thus if we have either $c^{11}_{d_1-i}=1$ with $0<i<\delta$ or $c^{21}_{d_1-i}=1$ with $i<\delta-r+1$, then $s_1$ overlaps itself in at least the last two symbols.  Since $B$ is simple, such an overlap must involve an overlap of $B$ with the end of $B_n$, which is to say that $s\widehat{s'}$ must occur in $B$.  Now switch this occurrence of $s\widehat{s'}$ in $B$ to $ss'$, leaving anything before it alone (and throwing away anything after it).  Any word $B'$ beginning in this fashion begins with $b$, but we claim it cannot begin with either $B$ or $B_n$, contrary to Proposition \ref{bigsummary}.  It clearly cannot begin with $B$ by constructions.   If the noted $s\widehat{s'}$ occurs before $ss'$ in $B$, then $B'$ cannot begin with $B_n$ since $B_n$ agrees with $B$ up to $ss'$.  On the other hand, if the noted $s\widehat{s'}$ occurs after $ss'$ in $B$, then $B'$ cannot begin with $B_n$ since $B'$ agrees with $B$ up to $ss'$.
\end{proof}

By Proposition \ref{bigsummary}, we can write the matrix of $T_k$ in a uniform way as 
\begin{equation}
	\left[\begin{array}{c|ccccc|ccccc}
		T & \0 & \0 &  \cdots & \0 & \mathbf{e}_{a_{d_1+1}} & \0 & \0 & \cdots & \0 & \gamma\mathbf{e}_{b_{d_2+1}}\\ \hline
		\0 & 0 & 0 & \cdots & 0 & 0 & 			0 &  & \cdots &  0&  \\
		\0 & 1 & 0 &  \cdots & 0 & 0 &     		 & & &  &  \\
		\0 & 0 & 1 &  & 0& 0 &   				\vdots & &  & \vdots&\mathbf{f} \\
		\vdots & \vdots &  & \ddots & & &	 		& & & & \\
		\0 & 0 & 0 &  & 1 & 0 &     				0 & & \cdots & 0&\\ \hline
		
		\0 & 0 &&\cdots&& 0&			0 & 0 & \cdots & 0 & 0 \\
		\0 & 	&&&&&			1 & 0 &  \cdots & 0 & 0 \\
		\0 & \vdots &&&&\vdots&			0 & 1 &  & 0 & 0     \\
		\vdots & 	&&&&&		\vdots &  & \ddots & & \\
		\0 & 0 &&\cdots&&0&			0 & 0 &  & 1 & 0 
		\end{array}\right]
\end{equation}
where $\mathbf{f}$ is the column vector with $\alpha_{i_m}$ at position $i_m$ and $0$s elsewhere.  In case $d_2=h_2$, we have $\gamma=1$ and $\mathbf{f}=0$.

The next task is to determine the matrix of $E_{w_1}+E_{w_2}$.  Generalizing the one-word case, this matrix tracks overlaps between $w_1$ and $w_2$.  We adopt the following notation: $c_i^{11}$ will track self-overlaps of $w_1$, $c_i^{22}$ will track self-overlaps of $w_2$, $c_i^{21}$ will track overlaps of the beginning $w_2$ with the end of $w_1$, and $c_i^{12}$ will track overlaps of the beginning of $w_1$ with the end of $w_2$.   Explicitly, we have, for $0\leq i\leq d_1$ and $0\leq j\leq d_2$, \\\noindent
$c^{11}_{d_1-i} = 1$ if $$a_0a_1\cdots a_{k-i} = a_{i}\cdots a_k,$$ $c^{21}_{d_1-i}=1$ if $$b_0\cdots b_{k-i} = a_{i}\cdots a_k,$$  $c^{12}_{d_2-j}=1$ if $$a_0\cdots a_{k-j} = b_{j}\cdots b_k,$$  $c^{22}_{d_2-j} = 1$ if $$b_0\cdots b_{k-j} = b_{j}\cdots b_k$$ and all other coefficients are $0$.  The matrix of $E_{w_1}+E_{w_2}$ with respect to the basis (\ref{basis}) is thus
\begin{equation}
	\left[\begin{array}{c|ccccc|ccccc}
		\0 & \0 & \0 &  \cdots & \0 & \0 										& \0 & \0 & \cdots & \0 & \0 \\ \hline
		\mathbf{e}_{a_0}^T & c^{11}_{d_1-1} & c^{11}_{d_1-2} & \cdots & c^{11}_1 & c^{11}_0 & 			c^{12}_{d_2-1} & c^{12}_{d_2-2} & \cdots & c^{12}_1 & c^{12}_0 \\
		\0 & 0 & 0 &  \cdots & 0 & 0 &     		 							 0 & 0 & \cdots & 0 & 0 \\
		 \vdots &  &  & \vdots & & &	 								& & \vdots & & \\
		\0 & 0 & 0 & \cdots & 0 & 0 &     											0 & & \cdots & & 0\\ \hline
		
		\mathbf{e}_{b_0}^T & c^{21}_{d_1-1} & c^{21}_{d_1-2} & \cdots & c^{21}_1 & c^{21}_0 & 			c^{22}_{d_2-1} & c^{22}_{d_2-2} & \cdots & c^{22}_1 & c^{22}_0 \\
		\0 & 0 & 0 &  \cdots & 0 & 0 &     		 							 0 & 0 & \cdots & 0 & 0 \\
		 \vdots &  &  & \vdots & & &	 								& & \vdots & & \\
		\0 & 0 & 0 & \cdots & 0 & 0 &     											0 & & \cdots & & 0\\ 
		\end{array}\right]
\end{equation}

Subtracting, we see that the matrix of $T_k\langle w_1,w_2\rangle-t$ is 
\begin{equation*}
	\left[\begin{array}{c|ccccc|ccccc}
		T-t& \0 & \0 &  \cdots & \0 & \mathbf{e}_{a_{d_1+1}} & \0 & \0 & \cdots & \0 & \gamma\mathbf{e}_{b_{d_2+1}}\\ \hline
		-\mathbf{e}_{a_0}^T & -c^{11}_{d_1-1}-t & -c^{11}_{d_1-2} & \cdots & -c^{11}_1 & -c^{11}_0 & 	-c^{12}_{d_2-1} & -c^{12}_{d_2-2} & \cdots & -c^{12}_1 & -c^{12}_0 \\
		\0 & 1 & -t &  \cdots & 0 & 0 &     		 &  &  &  &  \\
		\0 & 0 & 1 &  & 0& 0 &   				\0 & \0 & \cdots & \0 & \mathbf{f} \\
		\vdots & \vdots &  & \ddots & & &	 		& & & & \\
		\0 & 0 & 0 &  & 1 & -t &     				 & &  & & \\ \hline
		
		-\mathbf{e}_{b_0}^T & -c^{21}_{d_1-1} & -c^{21}_{d_1-2} & \cdots & -c^{21}_1 & -c^{21}_0 & 	-c^{22}_{d_2-1}-t & -c^{22}_{d_2-2} & \cdots & -c^{22}_1 & -c^{22}_0 \\
		\0 & 0&0&\cdots&0&0&		1 & -t &  \cdots & 0 & 0 \\
		\0 & \vdots &&&&\vdots&			0 & 1 &  & 0 & 0     \\
		\vdots & 	&&&&&		\vdots &  & \ddots & & \\
		\0 & 0 &&\cdots&&0&			0 & 0 &  & 1 & -t 
		\end{array}\right]
\end{equation*}
Using elementary row and column operations and collapsing some determinant one blocks, we find that the determinant of this matrix coincides (up to a sign) with that of
\begin{equation}\label{matrix}
\left[\begin{array}{c|c|c}
	T-t & \mathbf{e}_{a_{d_1+1}} & \gamma\mathbf{e}_{b_{d_2+1}} \\ \hline
	-\mathbf{e}^T_{a_0} & -p_{11}(t) & -p_{12}(t)+f_{11}(t) \\ \hline
	-\mathbf{e}^T_{b_0} & -p_{21}(t) & -p_{22}(t)+f_{21}(t)
\end{array}\right]
\end{equation}
where 
\begin{multicols}{2}
\noindent\begin{eqnarray*}
p_{11}(t) &=& \sum_{i=0}^{d_1} c_{d_1-i}^{11}t^{d_1-i} \\
p_{22}(t) &=& \sum_{j=0}^{d_2} c_{d_2-j}^{22}t^{d_2-j} 
\end{eqnarray*}
\begin{eqnarray*}
p_{21}(t) &=& \sum_{i=1}^{d_1} c_{d_1-i}^{21}t^{d_1-i} \\
p_{12}(t) &=& \sum_{j=1}^{d_2} c_{d_2-j}^{12}t^{d_2-j} 
\end{eqnarray*}
\end{multicols}
\noindent are the {\it correlation polynomials}.  The polynomials $f_{11}$ and $f_{21}$ are slightly more complicated to write down.  For each index $i\in \{0,\dots, d_1-1\}$, consider the truncated and shifted correlation polynomials $$q_{11}^{i}(t) = t^{i}+c^{11}_{d_1-1}t^{i-1}+\cdots + c^{11}_{d_1-i}$$ and $$q_{21}^{i}(t) = c^{21}_{d_1-1}t^{i-1}+\cdots + c^{21}_{d_1-i}$$ where we set $q_{21}^0(t)=0$ by convention.   Then we have $$f_*(t) = \sum_{m=0}^n\alpha_{i_m}q_*^{i_m}(t).$$

\section{Two Words: Bounding the Correlation Polynomials}
	For the moment, let $w=a_0a_1\cdots a_k$ be a general admissible word with fundamental period $p$.  Let $w'=b_0b_1\cdots b_k$ be another admissible word and let $i_0$ denote the smallest non-negative integer with 
	\begin{equation}\label{overlap1}
	b_0\cdots b_{k-i_0}= a_{i_0}\cdots a_k
	\end{equation}
	 	In practice $w$ and $w'$ will be taken from the forbidden words $\{w_1,w_2\}$.	In particular, we may have $w=w'=w_1$, for example, in which case $i_0=0$.  
\begin{lemm}
	Let $i$ be a positive integer with 
	\begin{equation}\label{overlap2}
	b_0\cdots b_{k-i}= a_{i}\cdots a_k
	\end{equation}
	 Then either $p|(i-i_0)$ or $i\geq k+2-p$.
\end{lemm}
\begin{proof}
	Combining (\ref{overlap1}) and (\ref{overlap2}) forces a self-overlap of the truncated word $\eta^{i_0}w$ in $k-i+1$ symbols via a shift of length $i-i_0$.  If $k-i+1\geq p$, then this retains a full period of $w$ and Lemma \ref{Bcopy} implies $p|(i-i_0)$.
\end{proof}
Let $D\leq k-1$ be a positive integer (which in practice will depend on both $w$ and $w'$) and consider the polynomial 
\begin{equation}\label{polyform1}
P(t) = \sum_{i=0}^D c_{D-i}t^{D-i}
\end{equation}
 where $c_{D-i} = 1$ if (\ref{overlap2}) holds and $0$ otherwise.  By the lemma above, the terms in this polynomial are either $t^{D-(i_0+pm)}$ for some $m\leq M:=\lfloor\frac{D-i_0}{p}\rfloor$ or $t^{D-i}$ with $$D-i\leq D-(k+2-p)\leq k-1-(k+2-p) = p-3.$$  As a result, we have 
\begin{equation}\label{polyform2}
P(t) = t^{D-i_0}+t^{D-(i_0+p)}+\cdots +t^{D-(i_0+Mp)}+\psi(t) = t^{D-i_0}\frac{t^{p} - t^{-Mp}}{t^p-1}+\psi(t)
\end{equation}
where $\psi(t)$ has degree at most $p-3$.

In what follows we fix a number $\rho>1$ and take $t$ to be a complex number with $|t|\geq\rho$ and seek bounds on these polynomials.  Specifically, we will proceed as in the one-word case and use the expressions in (\ref{polyform2}) to write $P(t)$ as the sum $$P(t) = D(t)+E(t)$$ of a dominant term and an error term.  In the ``large period'' cases, we will simply take $D(t) = t^{D-i_0}$ and let $E(t)$ be the remaining terms.  Reasoning as in Section \ref{oneword}, we have
\begin{equation}\label{errorlargep}
|D(t)|\leq |t|^D\ \ \ \ \ \mbox{and}\ \ \ \ \ |E(t)|\leq \frac{1}{\rho-1}|t|^{D-p+1}
\end{equation}
In the ``small period'' cases, we will take $$D(t) =  t^{D-i_0}\frac{t^{p} - t^{-Mp}}{t^p-1}$$ and let $E(t)=\psi(t)$.   Here we have 
\begin{equation}\label{errorsmallp}
|D(t)|\leq |t|^D\frac{\rho+1}{\rho-1}\ \ \ \ \ \mbox{and}\ \ \ \ \ |E(t)|\leq \frac{1}{\rho-1}|t|^{p-2}
\end{equation}

We now return to the situation of forbidding the pair of words $w_1=a_0\cdots a_k$ and $w_2=b_0\cdots b_k$.  
In the notation of Lemma \ref{detlemma}, write summand of the characteristic polynomial of the perturbed subshift corresponding to $S=T=\emptyset$ as $\chi_T(t)\Delta(t)$.  In other words, let $\Delta(t)$ denote the determinant of the bottom-right $2\times 2$ submatrix of (\ref{matrix}).   A lower bound on $\Delta(t)$ will be critical to our bound on the perturbed eigenvalue.  
\begin{prop}\label{realdeltabound}
Let $\rho>1$.  There exists a positive constant $D$ such that $$|\Delta(t)|\geq D|t|^{d_1+d_2}$$ holds on $|t|\geq \rho$ for all pairs of words with $d_1$ and $d_2$ sufficiently large.
\end{prop}

The following is Proposition \ref{realdeltabound} in case $d_2=h_2$, and will be an important step toward proving the proposition in general.
\begin{prop}\label{deltabound}
	Let $\rho>1$.  There exists a positive constant $D_0$ such that $$|p_{11}(t)p_{22}(t)-p_{12}(t)p_{21}(t)|\geq D_0|t|^{d_1+d_2}$$ holds on $|t|\geq \rho$ for all pairs of words with $d_1$ and $d_2$ sufficiently large.
\end{prop}
\begin{proof}
Let $p_1$ and $p_2$ denote the fundamental periods of $w_1$ and $w_2$, respectively.  Let $i_1$ denote the smallest positive integer $i$ satisfying (\ref{overlap2}), and let $i_2$ denote the analogous integer with the roles of $w_1$ and $w_2$ reversed.  We will arrive at our bound by dividing into subcases according to the size of $p_1$ and $p_2$ relative to $d_1$ and $d_2$.  In each case, we write the correlation polynomial $p_*(t)$ as a dominant term $D_*(t)$ plus an error term $E_*(t)$ as described above.  We then bound from \emph{below} the dominant contribution to $\Delta(t)$, namely $$D_{11}(t)D_{22}(t)-D_{12}(t)D_{21}(t)$$ and bound from \emph{above} the error contribution, which is
\begin{equation}\label{error}
E_{11}(t)D_{22}(t)+D_{11}(t)E_{22}(t)+E_{11}(t)E_{22}(t)-E_{12}(t)D_{21}(t)-D_{12}(t)E_{21}(t)-E_{12}(t)E_{21}(t).
\end{equation}
We proceed as in the one-word case by first fixing cutoffs $\alpha_1,\alpha_2\in(0,1)$ and using it to break into cases.

\bigskip
\noindent\framebox{Case $p_1\geq \alpha_1 d_1$ and $p_2\geq \alpha_2 d_2$}
\smallskip

\noindent This is the simplest case, as we may simply regard the leading term of the $p_*(t)$ in (\ref{polyform1}) as the dominant term and the rest as the error.  Thus $\Delta(t)$ is equal to $$t^{d_1+d_2} - t^{d_1+d_2-i_1-i_2}$$ plus an error term as in (\ref{error}).  Using (\ref{errorlargep}), we see that each term in the error is bounded by a constant multiple of either $$|t|^{(1-\alpha_1)d_1+d_2+1}\ \ \  \mbox{or}\ \ \  |t|^{d_1+(1-\alpha_2)d_2+1}\ \ \ \mbox{or}\ \ \ |t|^{(1-\alpha_1)d_1+(1-\alpha_2)d_2+2}$$  For the dominant term, note that $$|t^{d_1+d_2}-t^{d_1+d_2-i_1-i_2}|\geq |t|^{d_1+d_2}(1-|t|^{-i_1-i_2})\geq |t|^{d_1+d_2}(1-\rho^{-2})$$ since we take $|t|\geq \rho>1$.

Putting this together (and giving the constants generic names), we see that 
\begin{eqnarray*}
|\Delta(t)| &\geq& D_0'|t|^{d_1+d_2} - B_1|t|^{(1-\alpha_1)d_1+d_2+1} - B_2|t|^{d_1+(1-\alpha_2)d_2+1}-B_3|t|^{(1-\alpha_1)d_1+(1-\alpha_2)d_2+2} \\
&=& |t|^{d_1+d_2}\left( D_0' - B_1 |t|^{-\alpha_1 d_1+1}-B_2 |t|^{-\alpha_2 d_2+1} - B_3|t|^{-\alpha_1 d_1-\alpha_2 d_2+2}\right) \\ &\geq & |t|^{d_1+d_2}D_0
\end{eqnarray*}
for $d_1$ and $d_2$ sufficiently large, where $D_0$ can be taken to be any positive real number less than $D_0'$.

\bigskip
\noindent\framebox{Case $p_1\leq \alpha_1 d_1$ and $p_2\leq \alpha_2 d_2$ }
\smallskip

\noindent This case is the opposite extreme, wherein we use (\ref{polyform2}) with $\psi(t)$ as the error.  Let 
$$M_{11} = \left\lfloor \frac{d_1}{p_1}\right\rfloor \hspace{.5in}
M_{21} = \left\lfloor \frac{d_1-i_1}{p_1}\right\rfloor \hspace{.5in}
M_{12} = \left\lfloor \frac{d_2-i_2}{p_2}\right\rfloor \hspace{.5in}
M_{22} = \left\lfloor \frac{d_2}{p_2}\right\rfloor$$
We see that $\Delta(t)$ is equal to
\begin{eqnarray}\label{deltasmallp}
\lefteqn{
\frac{t^{d_1+d_2}(t^{p_1}-t^{-M_{11}p_1})(t^{p_2}-t^{-M_{22}p_2}) - t^{d_1+d_2-i_1-i_2}(t^{p_1}-t^{-M_{21}p_1})(t^{p_2}-t^{-M_{12}p_2})}{(t^{p_1}-1)(t^{p_2}-1)} }&& \\ &=&\nonumber
t^{d_1+d_2}\frac{(t^{p_1}-t^{-M_{11}p_1})(t^{p_2}-t^{-M_{22}p_2})}{(t^{p_1}-1)(t^{p_2}-1)}\left[1- t^{-i_1-i_2}\left(\frac{t^{p_1}-t^{-M_{21}p_1}}{t^{p_1}-t^{-M_{11}p_1}}\right) \left(\frac{t^{p_2}-t^{-M_{12}p_2}}{t^{p_2}-t^{-M_{22}p_2}} \right) \right]
\end{eqnarray}
plus the error term (\ref{error}), in which each term is bounded by a constant multiple of either $$|t|^{\alpha_1 d_1+d_2}\ \ \  \mbox{or}\ \ \  |t|^{d_1+\alpha_2 d_2}\ \ \ \mbox{or}\ \ \ |t|^{\alpha_1 d_1+\alpha_2 d_2}$$ by (\ref{errorsmallp}).

As for the main term, note that $$\left|\frac{t^{p_1}-t^{-M_{11}p_1}}{t^{p_1}-1}\right|\geq \frac{|t|^{p_1} - 1}{|t|^{p_1}+1}\geq \frac{\rho-1}{\rho+1}$$ and similarly $$\left|\frac{t^{p_2}-t^{-M_{22}p_2}}{t^{p_2}-1}\right|\geq \frac{\rho-1}{\rho+1},$$ so the term outside of the brackets has magnitude $$\geq |t|^{d_1+d_2}\left(\frac{\rho-1}{\rho+1}\right)^2.$$  To bound from below the quantity in brackets, we bound from above the subtracted term.  First note that 
 \begin{eqnarray*}
 \left|\frac{t^{p_1}-t^{-M_{21}p_1}}{t^{p_1}-t^{-M_{11}p_1}}\right| = \left|1+ \frac{1-t^{(M_{11}-M_{21})p_1}}{t^{M_{11}p_1+p_1}-1}\right| &\leq& 1+\frac{1+|t|^{i_1+p_1}}{|t|^{d_1}-1}\\ &= &1+\frac{|t|^{-d_1}+|t|^{i_1+p_1-d_1}}{1-|t|^{-d_1}} \\&\leq& 1+\frac{|t|^{i_1+p_1-d_1}+|t|^{i_1+p_1-d_1}}{1-\rho^{-1}} 
 \\ &=& 1+ \frac{2\rho}{\rho-1}|t|^{i_1+p_1-d_1}
 \end{eqnarray*}
  and similarly for the second factor.  Thus we have 
\begin{eqnarray*}
\lefteqn{\left| t^{-i_1-i_2}\left(\frac{t^{p_1}-t^{-M_{21}p_1}}{t^{p_1}-t^{-M_{11}p_1}}\right) \left(\frac{t^{p_2}-t^{-M_{12}p_2}}{t^{p_2}-t^{-M_{22}p_2}}\right)\right| } && \\ &\leq& 
|t|^{-i_1-i_2}\left(1+ \frac{2\rho}{\rho-1}(|t|^{i_1+p_1-d_1}+|t|^{i_2+p_2-d_2})+\left(\frac{2\rho}{\rho-1}\right)^2|t|^{i_1+i_2+p_1+p_2-d_1-d_2}\right) \\ &\leq & 
|t|^{-i_1-i_2}+\frac{2\rho}{\rho-1}(|t|^{-i_2-(1-\alpha_1)d_1}+|t|^{-i_1-(1-\alpha_2)d_2})+\left(\frac{2\rho}{\rho-1}\right)^2|t|^{-(1-\alpha_1)d_1-(1-\alpha_2)d_2}\\ &\leq & 
\rho^{-2}+\frac{2\rho}{\rho-1}(\rho^{-i_2-(1-\alpha_1)d_1}+\rho^{-i_1-(1-\alpha_2)d_2})+\left(\frac{2\rho}{\rho-1}\right)^2\rho^{-(1-\alpha_1)d_1-(1-\alpha_2)d_2}
\end{eqnarray*}
For $d_1$ and $d_2$ sufficiently large, the second and third summands can be made arbitrarily small, and in particular collectively smaller than $(1-\rho^{-2})/2$.  This entire expression is then $$\leq \rho^{-2}+\frac{1-\rho^{-2}}{2} = \frac{1+\rho^{-2}}{2}$$  Returning to (\ref{deltasmallp}), the bracketed term is then $$\geq 1-\frac{1+\rho^{-2}}{2} = \frac{1-\rho^{-2}}{2}$$ and we conclude the dominant contribution to $\Delta(t)$ is $$\geq |t|^{d_1+d_2}\left(\frac{\rho-1}{\rho+1}\right)^2\frac{1-\rho^{-2}}{2}$$ for $d_1,d_2$ sufficiently large. 

Arguing as in the end of the previous case, these estimates for the dominant and error contributions to $\Delta(t)$ imply that $|\Delta(t)|$ is at least a constant multiple of $|t|^{d_1+d_2}$ for $d_1$ and $d_2$ sufficiently large.

\bigskip
\noindent\framebox{Case  $p_1\geq \alpha_1 d_1$ and $p_2\leq \alpha_2 d_2$}
\smallskip

\noindent This is a hybrid case where we take the leading term in $p_{11}$ and $p_{21}$ as the dominant term and the rest as error, while for $p_{12}$ and $p_{22}$ we use (\ref{polyform2}) with $\psi(t)$ as the error.  We see that $\Delta(t)$ is equal to $$\frac{t^{d_1+d_2}(t^{p_2}-t^{-M_{22}p_2})-t^{d_1+d_2-i_1-i_2}(t^{p_2} - t^{-M_{12}p_2})}{t^{p_2}-1}$$ plus an error given by (\ref{error}), wherein each term is bounded by either $$|t|^{(1-\alpha_1)d_1+d_2+1}\ \ \ \mbox{or}\ \ \ |t|^{d_1+\alpha_2 d_2}\ \ \ \mbox{or}\ \ \ |t|^{(1-\alpha_1)d_1+\alpha_2 d_2+1}.$$  The dominant term can be rewritten as $$\frac{t^{d_1+d_2}(t^{p_2}-t^{-M_{22}p_2})}{t^{p_2}-1}\left[ 1 - t^{-i_1-i_2}\frac{t^{p_2}-t^{-M_{12}p_2}}{t^{p_2} - t^{-M_{22}p_2}}\right]$$ and we can proceed as in the previous part to see that the absolute value of this is $$\geq |t|^{d_1+d_2}\left(\frac{\rho-1}{\rho+1}\right)\frac{1-\rho^{-2}}{2}$$ for $d_1$ and $d_2$ sufficiently large, and again arrive at the desired lower bound on $|\Delta(t)|$.  

\bigskip
\noindent\framebox{Case $p_1\leq \alpha_1 d_1$ and $p_2\geq \alpha_2 d_2$ }
\smallskip

\noindent This case is handled precisely as the previous one, but with the roles of $w_1$ and $w_2$ reversed.
\end{proof}

In case $d_2<h_2$, we must contend with the polynomials $f_{11}$ and $f_{21}$.  Note that, for $*\in\{11,21\}$, we have $$p_*(t) = t^{d-i}q^i_*(t)+r^i_*(t)$$ where $$r^i_*(t) = c^*_{d-i-1}t^{d-i-1}+\cdots+c^*_0.$$  
We have
\begin{equation}\label{newdelta}
\Delta(t) = p_{11}(p_{22}-f_{21})-p_{21}(p_{12}- f_{11}) = (p_{11}p_{22}-p_{12}p_{21}) - (p_{11}f_{21} - p_{21}f_{11})
\end{equation}
wherein the first part of the last expression has been bounded below in Proposition \ref{deltabound}.  The second part is $$p_{11}f_{21} - p_{21}f_{11} =\sum_{m=1}^n \alpha_{i_m}(p_{11}q_{21}^{i_m}-p_{21}q_{11}^{i_m})$$ and 
$$p_{11}q_{21}^{i_m}-p_{21}q_{11}^{i_m} = (t^{d-i_m}q_{11}^{i_m}+r_{11}^{i_m})q_{21}^{i_m} - (t^{d-i_m}q_{21}^{i_m}+r_{21}^{i_m})q_{11}^{i_m} = 
r_{11}^{i_m}q_{21}^{i_m}-r_{21}^{i_m}q_{11}^{i_m}$$  

Note that $q_*^{i_m}$ has degree at most $i_m$ and $r_*^{i_m}$ has degree at most $d_1-i_m-1$, and each of these polynomials have coefficients in $\{0,1\}$.  This yields the simple bound $$|r_{11}^{i_m}q_{21}^{i_m}-r_{21}^{i_m}q_{11}^{i_m}|\leq |r_{11}^{i_m}q_{21}^{i_m}|+|r_{21}^{i_m}q_{11}^{i_m}|\leq \frac{2}{(\rho-1)^2}|t|^{d_1+1}$$  In Cases (B) and (C) of Proposition \ref{bigsummary} where we have at most two $i_m$ to contend with, this gives the bound $$|p_{11}f_{21}-p_{21}f_{11}|\leq \frac{4}{(\rho-1)^2}|t|^{d_1+1}$$

Case (A) presents a deeper challenge, wherein we will rely critically on Corollary \ref{smalloverlap}, which applies perhaps after switching the words $w_1,w_2$.  This corollary furnishes the bounds
\begin{eqnarray*}
|r_{11}^i| &\leq& \left\{\begin{array}{ll} \frac{1}{\rho-1}|t|^{d_1-\delta+1} & i+1<\delta \\  \frac{1}{\rho-1}|t|^{d_1-i}& i+1\geq \delta
\end{array}\right. \\
|r_{21}^i| &\leq& \left\{\begin{array}{ll} \frac{1}{\rho-1}|t|^{d_1-\delta+r} & i+1<\delta-r+1 \\  \frac{1}{\rho-1}|t|^{d_1-i}& i+1\geq \delta-r+1
\end{array}\right.
\end{eqnarray*}
Now we multiply by the simple bound $$|q_*^i|\leq \frac{1}{\rho-1}|t|^{i+1}$$ and sum over $i$.
We have
\begin{eqnarray*}
\sum_m |r_{11}^{i_m}q_{21}^{i_m}| &\leq& \sum_{i=0}^{\delta-2}\frac{1}{(\rho-1)^2}|t|^{d_1-\delta+2+i}+ \sum_{i=\delta-1}^{d_1-1}\frac{1}{(\rho-1)^2}|t|^{d_1+1}\\	&\leq & \frac{1}{(\rho-1)^3}|t|^{d_1+1}+\frac{d_1-\delta+1}{\rho-1}|t|^{d_1+1}
\end{eqnarray*}
Since $$d_1-\delta+1=d_1-(h_2-d_2)+1 \leq d_2+r$$ this expression is bounded above by a constant multiple of $d_2|t|^{d_1+1}$ for $d_2$ sufficnelty large.  The same reasoning applies to the sum of $|r_{21}^{i_m}q_{11}^{i_m}|$ to arrive at the same bound (with perhaps different constants).

In summary, we have in all cases $$|p_{11}f_{21}-p_{21}f_{11}|\leq Ed_2|t|^{d_1+1}$$ for $d_2$ sufficiently large and some constant $E$.  Combining this with Proposition \ref{deltabound}, we see that $$|\Delta(t)|\geq D_0|t|^{d_1+d_2}-Ed_2|t|^{d_1+1} = |t|^{d_1+d_2}(D_0-Ed_2|t|^{-d_2+1})\geq D|t|^{d_1+d_2}$$ for some constant $D$, as long as $d_1$ and $d_2$ are both sufficiently large.  This completes the proof of Proposition \ref{realdeltabound}.

\section{Two Words: Bounding the Perturbed Eigenvalue}\label{results}

Let $X(t)$ denote the characteristic polynomial of the perturbed subshift $\TkC$.  Applying Lemma \ref{detlemma} to (\ref{matrix}), we see 
\begin{equation}\label{pertcharpoly}
X(t) = \Delta(t)\chi_T(t)+M(t)
\end{equation}
 where $M(t)$ is a signed sum of minors of $T-t$ times correlation polynomials.  For example, if $d_2=h_2$ with $a_0\neq b_0$ and $a_{d_1+1}\neq a_{d_2+1}$, then $$M(t) = \pm p_{22}M_{a_{d_1+1};a_0}\pm p_{21}M_{b_{d_2+1};a_0}\pm p_{12}M_{a_{d_1+1};b_0}\pm p_{11}M_{b_{d_2+1};b_0}+M_{a_{d_1+1},b_{d_2+1};a_0,b_0}$$ where $M_{A;B}$ denotes the minor of $T-t$ obtained by deleting the $A$ rows and the $B$ columns.

\begin{lemm}\label{Mbound}
	Suppose that  $1<\rho<\lambda_0$.  There exists constants $M_0,M_1,M_2$ such that $$|M(t)|\leq M_0+M_1 |t|^{d_1}+M_2|t|^{d_2}$$ for all $t$ in the annulus $\rho\leq |t|\leq \lambda_0$.
\end{lemm}
\begin{proof}
	Since the collection of minors of $T-t$ is finite, they are collectively bounded above by a single constant on the compact annulus $\rho\leq|t|\leq\lambda_0$.  It remains to bound the polynomials $p_*$ and $f_*$.  Looking merely at degree and coefficients gives the simple bounds $$|p_{11}(t)|, |p_{21}(t)|\leq \frac{|t|^{d_1+1}}{\rho-1}\ \ \ \ \ \mbox{and}\ \ \ \ \ |p_{12}(t)|, |p_{22}(t)|\leq \frac{|t|^{d_2+1}}{\rho-1}$$ as usual.  Similarly, we have $$|f_*|\leq \sum_m|q_*^{i_m}(t)|\leq\sum_m \frac{|t|^{i_m+1}}{\rho-1}\leq \frac{|t|}{\rho-1}\left(1+|t|+\cdots+|t|^{d_1-1}\right)\leq \frac{1}{(\rho-1)^2}|t|^{d_1+1}$$ for $*\in\{11,21\}$.  
\end{proof}
\begin{prop}\label{qual}
Suppose $\rho>1$ satisfies $|\lambda|<\rho<\lambda_0$ for all non-dominant eigenvalues $\lambda$ of $T$.	We have $$\lambda_1\geq\rho$$ for all pairs of admissible words $w_1,w_2$ of sufficient length.
\end{prop}
\begin{proof}
	By Lemma \ref{hgrows}, as $k\to \infty$ we have $d_1=h_1\to\infty$ as well.  By contrast, we have no control over $d_2$.  However, for large $d_2$ we can use Proposition \ref{realdeltabound}, while small $d_2$ ensures that $w_1$ and $w_2$ share a large common subword that we can exploit.

The hypothesis on $\rho$ implies that $\chi_T(t)$ is bounded below by a positive constant on $|t|=\rho$, so Lemma \ref{Mbound} and Proposition \ref{realdeltabound} imply that $$|M(t)|< |\Delta(t)\chi_T(t)|$$ for $d_1$ and $d_2$ sufficiently large.  Rouch\'e's Theorem implies that $X(t)$ and $\Delta(t)\chi_T(t)$ have the same number of roots in $|t|<\rho$ for such $d_1,d_2$.  Since these two polynomials have the same degree, it follows that they have the same number of roots with $|t|\geq \rho$, and in particular the largest root of $X(t)$ satisfies $\lambda_1\geq \rho$ if both $d_1$ and $d_2$ are sufficiently large.  

If $d_2\geq d_1/2$ and $k$ is sufficiently large, then the argument of the previous paragraph applies and yields the desired result.  On the other hand, if $d_2<d_1/2$, then by Proposition \ref{bigsummary}, $w_1$ and $w_2$ share a common subword of length $$\delta = h_2-d_2\geq k-r-\frac{1}{2}(k-1)=\frac{k}{2}-r+\frac{1}{2}\to\infty$$ as $k\to \infty$. Forbidding this common subword forbids both $w_1$ and $w_2$, so the desired inequality $\lambda_1\geq \rho$ follows from Proposition \ref{onewordqualbound} applied to the common subword.	
\end{proof}

Proposition \ref{qual} states in effect that the entropy of the perturbed shift can be made as close to that of the original shift by taking $k$ large enough.  As in the one-word situation, we can moreover bound the difference $|\lambda_1-\lambda_0|$.
\begin{theo}\label{mainB1}
Suppose that $\lambda_0>1$.  	There exist positive constants $C$ and $C'$ such that $$|\lambda_1-\lambda_0|\leq C\lambda_0^{-k}(1+C'\lambda_0^\delta)$$ for all pairs of admissible words $w_1,w_2$ with $d_1$ and $d_2$ sufficiently large.
\end{theo}
\begin{proof}
Choose $\rho>1$ with $|\lambda|<\rho<\lambda_0$ for all non-dominant eigenvalues $\lambda$ of $T$.
Write $\chi_T(t) = (t^s-\lambda_0^s)q(t)$ as in Section 1 and plug $t=\lambda_1$ into (\ref{pertcharpoly}) to conclude that 
	$$|\lambda_1^s-\lambda_0^s| = \frac{|M(\lambda_1)|}{|\Delta(\lambda_1)|\cdot |q(\lambda_1)|}$$  The hypotheses imply that $q(t)$ has no roots on $|t|\geq \rho$, and $|q(\lambda_1)|$ is therefore bounded below by a positive constant independent of $\lambda_1$.   By Lemma \ref{Mbound} and Proposition \ref{realdeltabound}, we have $$|M(t)|\leq M_0+M_1|t|^{d_1}+M_2|t|^{d_2}\leq M_1'|t|^{d_1}+M_2'|t|^{d_2}$$ and $$|\Delta(t)|\geq D|t|^{d_1+d_2}$$ for $d_1$ and $d_2$ sufficiently large.   Thus, since $d_1-d_2\leq (h_2-d_2)+(r-1)$, we have
	\begin{eqnarray*}
	|\lambda_1-\lambda_0|\leq |\lambda_1^s-\lambda_0^s|&\leq& A\lambda_1^{-d_1}+B\lambda_1^{-d_2} \\&=& A\lambda_1^{-d_1}(1+(B/A)\lambda_1^{d_1-d_2}) \\
		&\leq & A'\lambda_1^{-k}(1+(B/A)\lambda_1^{r-1}\lambda_1^{h_2-d_2}) \\ &\leq & A'\lambda_1^{-k}(1+C'\lambda_0^{h_2-d_2})
	\end{eqnarray*}
	since $\lambda_1\leq \lambda_0$ and $d_2\leq h_2$.    Now we proceed as in the proof of Proposition \ref{Lindmain} to replace $\lambda_1$ by $\lambda_0$ in this bound, which may increase the constant $A'$ somewhat to a constant $C$.
\end{proof}
The quantity $1+C'\lambda_0^\delta$ is a sort of correction factor that accounts for an overlap between our two words.  Note that, in Case (A) of Proposition \ref{bigsummary}, this factor is simply bounded above by a constant that can be effectively absorbed into $C$.   We can exploit the observation that the correction factor is controlled by the overlap in order to remove the growth restriction on $d_2$ in a similar fashion to the proof of Proposition \ref{qual}.   \begin{theo}\label{mainlast}
Suppose that $\lambda_0>1$.	There exists a positive constant $C$ such that  $$|\lambda_1-\lambda_0|\leq C\lambda_0^{-k/2}$$ for all pairs of sufficiently long admissible words $w_1,w_2$.
\end{theo}
\begin{proof}
	We consider two subcases given by $d_2\geq d_1/2$ and $d_2<d_1/2$.  In the former case, both $d_1$ and $d_2$ grow without bound as $k\to\infty$, so Theorem \ref{mainB1} applies and gives us the desired bound, since $$\delta=h_2-d_2\leq k-1-\frac{1}{2}(k-r)\leq \frac{k}{2}+\frac{r}{2}-1$$ so the correction factor is bounded by a constant multiple of $\lambda_0^{k/2}$.  In the latter case, there is a common subword of length $$\delta\geq k-r-\frac{d_1}{2}\geq \frac{k}{2}-r+\frac{1}{2}$$ and Theorem \ref{Lindmain} applied to this subword gives the desired bound for $k$ sufficiently large.
\end{proof}

\appendix

\section{A determinant lemma}

In this appendix, we state and prove a lemma that explains how to relate the characteristic polynomial of a perturbed subshift to that of the subshift itself.  Note that this section is self-contained and uses notation ({\it e.g.} $k$) independently of the rest of the paper.

Let $A$ and $B$ denote $n\times n$ and $m\times m$ matrices, respectively.  For $k=1,\dots,m$, let $1\leq i_k, j_k\leq n$ and consider the matrix
$$M = \left[\begin{array}{c|c} A & [\alpha_1\mathbf{e}_{i_1},\dots, \alpha_m\mathbf{e}_{i_m}] \\ \hline [\beta_1\mathbf{e}_{j_1},\dots, \beta_m\mathbf{e}_{j_m}]^T & B\end{array}\right]$$  Our goal in this section is to compute the determinant of $M$.

Let $S,T\subseteq\{1,\dots,m\}$ with $|S|=|T|$.  This determines a minor $B_{S,T}$ of $B$ obtained by deleting the rows indexed by $S$ and the columns indexed by $T$.  It also determines a minor of $A$ as follows.  Let 
\begin{eqnarray*}
J(S) &=& \{j_k\ |\ k\in S\} \\ I(T) &=& \{i_k\ |\ k\in T\}
\end{eqnarray*}
and let $A_{S,T}$ denote the minor of $A$ obtained by deleting rows indexed by $I(T)$ and the columns indexed by $J(S)$ provided that $|I(T)|=|J(T)| = |T|=|S|$, and set $A_{S,T}=0$ otherwise.  Finally, set $$\alpha_T = \prod_{k\in T}\alpha_{i_k}\ \ \ \ \ \mbox{and}\ \ \ \ \ \beta_S = \prod_{k\in S}\beta_{j_k}$$
\begin{lemm} \label{detlemma}
We have $$\det(M) = \sum_{S,T} \varepsilon_{S,T}\alpha_T\beta_S B_{S,T}A_{S,T}$$ where $\varepsilon_{S,T}\in\{\pm 1\}$ and the sum is over all pairs of subsets of $\{1,\dots,m\}$ of equal size.
\end{lemm}
\begin{proof}
For the purposes of this argument, we will re-index $A$ and $B$ disjointly by $\{a_1,\dots, a_n\}$ and $\{b_1,\dots, b_m\}$, respectively.  We have $$\det(M) = \sum_\sigma \mathrm{sgn}(\sigma)\prod_x M_{x,\sigma(x)}$$ where the sum is taken over all permutations of $\{a_1,\dots,a_n,b_1,\dots,b_m\}$.   Given $S,T\subseteq\{b_1,\dots, b_m\}$ of equal size, we can consider the collection of all such permutations satisfying $$\sigma(\{b_1,\dots,b_m\}\setminus S) = \{b_1,\dots,b_m\}\setminus T\ \ \ \ \ \mbox{and}\ \ \ \ \ \sigma(S)\cap \{b_1,\dots,b_m\}=\emptyset$$  For such $\sigma$ to give a nonzero contribution to $\det(M)$, we must have $\sigma(b_k)=a_{j_k}$ for $b_k\in S$.  In particular, this requires that $|J(S)|=|S|$.  Under such a $\sigma$, each $b_k\in T$ must be the image of some $a_i$.  In order to contribute nontrivially to the determinant, this $a_i$ must be $a_{i_k}$, and again we see $|I(T)|=|T|$.  The remaining elements of $\{a_1,\dots,a_n\}$ may be mapped to any element $\{a_1,\dots,a_n\}$ that is not among the $a_{j_k}$.  

Thus, to specify a one such $\sigma$ is precisely to specify a pair of bijections 
$$\{b_1,\dots,b_m\}\setminus S\to \{b_1,\dots, b_m\}\setminus T\ \ \ \ \ \mbox{and}\ \ \ \ \ \{a_1,\dots,a_n\}\setminus I(T)\to \{a_1,\dots,a_n\}\setminus J(S).$$
 On the other hand, if we {\it fix} a pair of identifications here for a particular pair $S$ and $T$,  then each such $\sigma$ corresponds to a pair of permutations - one for the index set of the submatrix of $A$ and one for that of $B$.  The sign of $\sigma$ is the product of the signs of these two permutations, up to a fixed sign that depends only on $S$, $T$, and the choice of identifications above.  Thus the net contribution to $\det(M)$ by such $\sigma$ is $\pm\alpha_T\beta_SB_{S,T}A_{S,T}$.  Now we need only note that {\it every} permutation $\sigma$ corresponds to a unique pair $S,T$, namely $$S = \{b_1,\dots,b_m\}\setminus\sigma^{-1}(\{b_1,\dots, b_m\})\ \ \ \ \ \mbox{and}\ \ \ \ \ T = \{b_1,\dots,b_m\}\setminus \sigma(\{b_1,\dots,b_m\}\setminus S)$$ and sum over pairs $S,T$.
\end{proof}

%\bibliography{psft}
%\bibliographystyle{smfplain}

\end{document}